\documentclass[AMA,STIX1COL]{WileyNJD-v2}

\articletype{Paper}%

\usepackage{comment}

\newtheorem{Thm}{Theorem}
\newtheorem{Cor}{Corollary}

\newtheorem{Lem}{Lemma}

\newtheorem{Def}{Definition}

\newtheorem{Exmp}{Example}
\newtheorem{rmk}{Remark}

\newcommand{\be}[1]{\begin{equation}\label{#1}}
	\newcommand{\ee}{\end{equation}}

\begin{document}

\title{Edge convex smooth interpolation curve networks with minimum $L_{\infty}$-norm of the second derivative}

\author{Krassimira Vlachkova}

\authormark{Krassimira Vlachkova}

\address{\orgdiv{Faculty of Mathematics and Informatics}, \orgname{Sofia University ``St. Kliment Ohridski''}, \orgaddress{\state{1164 Sofia}, \country{Bulgaria}}}

\corres{Krassimira Vlachkova,\\ \orgdiv{Faculty of Mathematics and Informatics},\\ \orgname{Sofia University ``St. Kliment Ohridski''},\\ \orgaddress{\state{1164 Sofia}, \country{Bulgaria}}.\\ \email{krassivl@fmi.uni-sofia.bg}}

\abstract{We consider the extremal problem of interpolation of convex scattered data in $\mathbb{R}^3$ by smooth edge convex curve networks with minimal $L_p$-norm of the second derivative for $1<p\leq\infty$. The problem for $p=2$ was set and solved by Andersson et al. (1995). Vlachkova (2019) extended the results in (Andersson et al., 1995) and solved the problem for $1<p<\infty$. The minimum edge convex $L_p$-norm network for $1<p<\infty$ is obtained from the solution to a system of nonlinear equations with coefficients determined by the data. The solution in the case $1<p<\infty$ is unique for strictly convex data. The corresponding extremal problem for $p=\infty$ remained open.
		Here we show that the extremal interpolation problem for $p=\infty$ always has a solution. We give a characterization of this solution. We show that a solution to the problem for $p=\infty$ can be found by solving a system of nonlinear equations in the case where it exists.
		}


\maketitle

\section{Introduction}\label{s0}

Interpolation of scattered data in $\mathbb{R}^3$ is an important problem in applied mathematics and finds applications in various fields such as computer graphics and animation, scientific visualization, medicine (computer tomography), automotive, aircraft and ship design, architecture, and many more. In general the  problem can be formulated as follows: Given a set of points $P_i=(x_i,y_i,z_i)$ $\in \mathbb{R}^3$,\ $i=1,\dots ,n$, find a bivariate function $F(x,y)$ defined in a certain domain $D$ containing points $V_i=(x_i,y_i)$, such that $F$ possesses continuous partial derivatives up to a given order
and $F(x_i,y_i)=z_i$.
Various methods for solving this problem were proposed and applied, see e.g. \cite{FH,FN,LF,MLLMPDS,D,Am,ALP,CG,DT}.
Nielson~\cite{N} proposed a three steps method for solving the problem as follows:

\smallskip
{\sl Step 1.} {\it Triangulation}. Construct a triangulation $T$ of
$V_i,\ i=1,\dots n$.

\smallskip
{\sl Step 2.} {\it Minimum norm network}. The interpolant $F$ and
its first order partial derivatives are defined on the edges of $T$
to satisfy an extremal property. The obtained minimum norm network
is a cubic curve network, i.~e. on every edge of $T$ it is a cubic
polynomial.

\smallskip
{\sl Step 3.} {\it Interpolation surface}. The obtained network is extended to $F$ by an
appropriate {\it blending method}.

\smallskip
Andersson et al. \cite{AEIV} focused on Step~2 of the
above method, namely the construction of the minimum norm network.
The authors gave a new proof of Nielson's result by using a different approach. They constructed a system
of simple linear curve networks called {\it basic curve networks} and then
represented the second derivative of the minimum norm network as a linear combination of these
basic curve networks. The new approach allows to consider and handle the case where the data are convex and we seek a convex interpolant. Andersson et al. formulate the corresponding extremal constrained interpolation problem of finding a minimum norm network that is convex along the edges of the triangulation. The extremal network is characterized as a solution to a nonlinear system of equations.
The authors propose a Newton-type algorithm for solving this type of systems. The validity and convergence of the algorithm were studied further in \cite{V3}. We note that the edge convex minimum norm network may not be globally convex.

Vlachkova~\cite{V} extended the results in \cite{AEIV} and solved the extremal unconstrained problem of interpolation of scattered data by minimum $L_p$-norms networks for  $1<p<\infty$. The minimum $L_p$-norm network for $1<p<\infty$ is obtained from the solution to a nonlinear (except for $p=2$ when it is linear) system of equations  with coefficients determined by the data. The solution in the case $1<p<\infty$ is unique.
The approach proposed in \cite{V} can not be applied to the case where $p=\infty$. Recently, Vlachkova~\cite{V5} established   the existence of a solution for $p=\infty$ of the same type as in the case where $1<p<\infty$.

The extremal constrained interpolation problem for $1<p<\infty$ was considered in \cite{V4} where the existence and the uniqueness of the solution in the case of strictly convex data were established and  a complete characterization of the solution using the basic curve networks  was presented. The approach used in \cite{V4} does not apply in the case where $p=\infty$ and the problem remained open.

In this paper we consider the extremal constrained interpolation problem for $p=\infty$ and prove the existence of a solution of a certain type. More precisely,  this solution on each edge of the underlying triangulation $T$ is a quadratic spline function with at most one knot. Moreover, we show that a solution to the problem for $p=\infty$ can be found by solving a system of nonlinear equations in the case where it exists.

This paper is an extended version of the results previously announced in \cite{V7} without proofs.

The paper is organized as follows. In Section~\ref{s1} we introduce notation, formulate the constrained extremal problem for interpolation by edge convex minimum $L_{\infty}$-norm networks, and present some related results. Our main results are obtained in Section~\ref{s2}. An example illustrating our main results is presented and discussed in Section~\ref{s3}. In the final Section~\ref{s4} we summarize our conclusions.

\section{Preliminaries and related work}\label{s1}

Let $n\geq 3$ be an integer and $P_i:=(x_i,y_i,z_i),\ i=1,\dots ,n$
be different points in $\mathbb{R}^3$. We call this set of points {\it
	data}. The data are {\it scattered}\footnote{Note that this definition of scattered data slightly misuses the commonly accepted meaning of the term. It allows data with some structure among points ${\bf v}_i$. We have opted to do this in order to cover 	all cases where our presentation and results are valid.}
if the projections
$V_i:=(x_i,y_i)$ onto the plane $Oxy$ are different and
non-collinear.
\begin{Def}\label{def1}
	A collection of non-overlapping, non-degenerate triangles in
	$\mathbb{R}^2$ is a {\it triangulation} of the points $V_i,\ i=1,\dots ,n$,
	if the set of the vertices of the triangles coincides with the set
	of the points $V_i,\ i=1,\dots ,n$.
\end{Def}

For a given triangulation $T$ there is a unique continuous function $L : D\rightarrow\mathbb{R}^1$ that is linear inside each of the triangles of $T$ and interpolates the data.
\begin{Def}\label{def2}
	Scattered data in $D$ are {\it convex} if there exists a triangulation $T$ of $V_i$ such that the corresponding function $L$ is convex. The data are {\it strictly convex} if they are convex and the gradient of $L$ has a jump discontinuity across each edge inside $D$.
\end{Def}
Hereafter we assume that the data are convex and $T$ is the associated triangulation of the points $V_i,\ i=1,\dots ,n$. Furthermore, for the sake of simplicity, we assume that $D$ is a convex polygonal domain and is formed by the union of all triangles in $T$. We note that in the case of strictly convex data the triangulation $T$ is unique and is obtained as the orthogonal projection of the convex hull of the data points $P_i$, $i=1,\dots ,n$, onto the plane $Oxy$, i.e. $T$ is the Delaunay triangulation of $V_i$.

The set of
the edges of the triangles in $T$ is denoted by $E$. If there is an
edge between $V_i$ and $V_j$, it will be referred to by
$e_{ij}$ or simply by $e$ if no ambiguity arises.

\begin{Def}\label{def3}
	{A {\it curve network} is a collection of real-valued univariate functions
		$\{f_e\}_{e\in E}$
		defined on the edges in $E$.}
\end{Def}
With any real-valued bivariate function $F$ defined on $D$ we naturally
associate the curve network defined as the restriction of $F$ on the edges
in $E$, i.~e. for $e=e_{ij}\in E$,
\begin{equation}\label{e0}
	\begin{split}
		f_e(t) := F\Bigl(\bigl(1-\frac{t}{\| e\|}\bigr)x_i+\frac{t}{\| e\|}\,x_j,\,
		\bigl(1-\frac{t}{\| e\|}\bigr)y_i+\frac{t}{\| e\|}\,y_j\Bigr), \\
		\text{where } 0\le t\le \| e\| \ \text{and} \ \|e\|=\sqrt{(x_i-x_j)^2+(y_i-y_j)^2}.
	\end{split}
\end{equation}
Furthermore, according to the context $F$ will denote either a
real-valued bivariate function or a curve network defined by
(\ref{e0}). For $p$, such that $1<p\leq\infty$, we introduce the
following class of {\it smooth interpolants}
$$
{\cal F}_p:=\{ F(x,y)\in C^1(D)\, | \, F(x_i,y_i)=z_i,\ i=1,\dots ,n,\ f'_e\in AC,\ f^{\prime\prime}_e\ \in
L_p,\ e\in E\},
$$
and the corresponding class of so-called {\it smooth interpolation edge convex curve networks}
\be{e100}
{\cal C}_p(E):=\left\{ F_{|E}=\{f_e\}_{e\in E}\ |\ F (x,y)\in {\cal
	F}_p,\ f_e''\geq 0,\ e\in E\right\},
\ee
where $C^1(D)$ is the class of bivariate functions defined in $D$ which possess continuous first order partial derivatives,
$AC$ is the class of univariate absolutely continuous functions defined in $[ 0,\| e\|]$, $L_p$ for $1<p<\infty$ is the class of univariate functions defined in $[ 0,\| e\|]$ whose p-th power of the absolute value is Lebesgue integrable, and $L_{\infty}$ is the class of the bounded univariate functions defined in $[ 0,\| e\|]$.

In \cite{AEIV} the following lemma is proved.

\begin{Lem}\label{lem1}
If the data are convex (strictly convex) then there exists a convex (strictly convex) function in the class $C^{\infty}(\mathbb{R}^2)$ interpolating the data points $P_i$, $i=1,\dots ,n$.
	\end{Lem}

From Lemma~\ref{lem1} we immediately obtain the following

\begin{Cor}\label{cor1}
	The class ${\cal C}_p(E)$ is nonempty.
	\end{Cor}

The smoothness of the interpolation curve network $F=\{f_e\}_{e\in E}\in {\cal C}_p(E)$
geometrically means that at each point $P_i$ there is a {\it tangent plane} to $F$,
where a plane is {\it tangent} to the curve network at the point $P_i$ if it contains the tangent vectors at $P_i$ of the curves incident to $P_i$.

Inner product and
$L_p$-norm are defined in ${\cal C}_p(E)$ by
\begin{eqnarray*}
	&&\langle F, G\rangle =\int _{E} FG=\sum _{e\in E}\int _{0}^{\|e\| }f
	_{e}(t)g_{e}(t)dt,
	\\
	&&\|F\|_p:= \left( \sum_{e\in E} \int_0^{\| e\|
	}|f_e(t)|^pdt\right)^{1/p}, \quad 1< p<\infty ,
	\\[1ex]
	&&\|F\|_{\infty}:= \max_{e\in E}\ \|f_e\|_{\infty}.
\end{eqnarray*}
where $F\in {\mathcal{C}}_{p}(E)$ and $G:=\{g_{e}\}_{e\in E}\in {\mathcal{C}}
_{p}(E)$.
We denote the networks of the
second derivative of $F$ by $F^{\prime\prime}:=\{f''_e\}_{e\in E}$ and
consider the following extremal problem:
$$({\bf P}_p)\quad  \mbox{\it\ Find\ }F^*\in {\cal C}_p(E)\  \mbox{\it\
	such that }\| F^{*\prime\prime}\|_p=\inf_{F\in {\cal C}_p(E)}\|
F''\|_p.
$$

Problem $({P}_p)$ is a generalization of the classical univariate extremal problem $(\tilde{P}_p)$ for interpolation of convex data in $\mathbb{R}^2$ by a univariate convex function with minimal $L_p$-norm of the second
derivative. Hornung~\cite{Horn} considered the problem $(\tilde{P}_p)$ for $p=2$. Iliev and Pollul~\cite{IP1} considered the case $p=\infty$ and proved that problem $(\tilde{P}_{\infty})$ has a quadratic spline solution characterized by the existence of a core interval on which the second derivative is the
positive part of a perfect spline and that all solutions agree on core intervals. Iliev and Pollul~\cite{IP2}, and Micchelli et al.~\cite{MSSW} studied in detail the problem  for $1<p<\infty$.\footnote{Iliev and Pollul~\cite{IP2}, and Micchelli et al.~\cite{MSSW} considered the more general problem of minimum $L_p$-norm of the $k$-th derivative which is non-negative, $k\geq 2$.}

For $i=1,\dots ,n$ let $m_i$ denote the degree of the vertex $V_i$,
i.~e. the number of the edges in $E$ incident to $V_i$.
Furthermore, let $\{ e_{ii_1},\dots ,e_{ii_{m_i}}\}$ be the edges
incident to $V_i$ listed in clockwise order around $V_i$. The first
edge $e_{ii_1}$ is chosen so that the coefficient
$\lambda_{1,i}^{(s)}$ defined below is not zero - this is always
possible. A {\it basic curve network} $B_{is}$ is defined on $E$ for
any pair of indices $is$, such that $i=1,\dots ,n$ and $s=1,\dots
,m_i-2$, as follows (see Figure~\ref{basic}):
$$
B_{is}:= \left\{
\begin{array}{ll}
	\lambda^{(s)}_{r,i}\left( 1-\frac{t}{\| e_{ii_{s+r-1}} \|
	}\right) & {\rm on}\ e_{ii_{s+r-1}},\ r=1,2,3,\\
	& \ \ \ 0\leq t\leq \| e_{ii_{s+r-1}}\| \\[1ex]
	0 & {\rm on\ the\ other\ edges\ of\ }E.
\end{array}
\right.
$$
\begin{figure}
	\centering
	\begin{minipage}{4.cm}
		\centering
		\includegraphics[width=1.\textwidth]{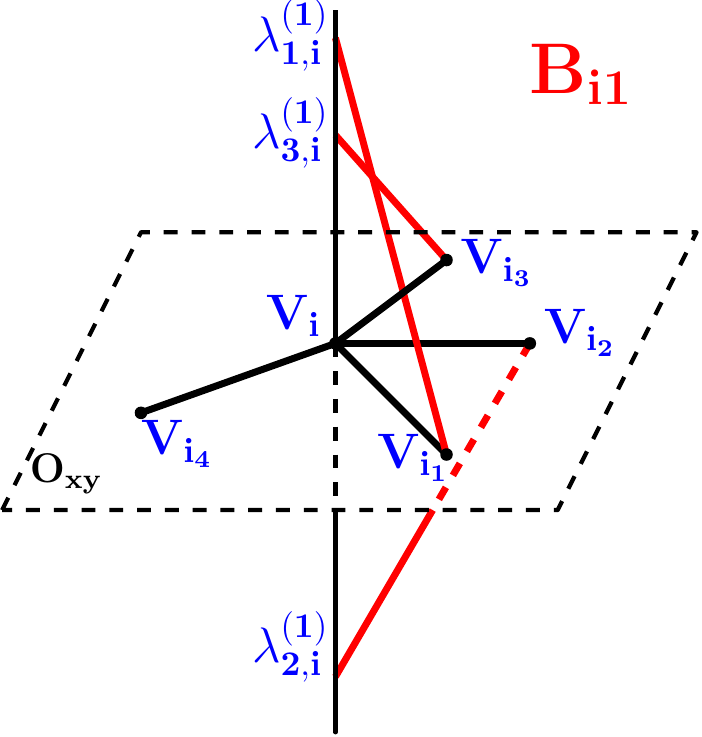}
	\end{minipage}
	~~~~~~~~~~~~~~~~~~~~~~~~~~
	\begin{minipage}{4.cm}
		\centering
		\includegraphics[width=1.\textwidth]{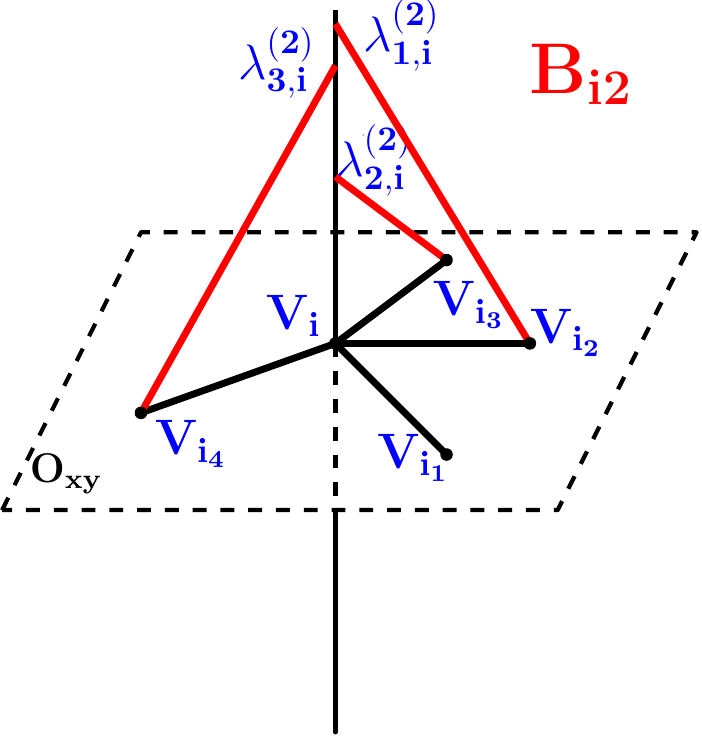}
	\end{minipage}
	\caption{The basic curve networks for vertex $V_i$, $\mbox{deg}(V_i)=4$}\label{basic}
\end{figure}

The coefficients $\lambda_{r,i}^{(s)},\ r=1,2,3$, are
uniquely determined to sum to one and to form a zero linear
combination of the three unit vectors along the edges
$e_{ii_{s+r-1}}$ starting at $V_i$. More precisely, $\lambda_{r,i}^{(s)},\ r=1,2,3$, are the unique solution to the linear system
\be{e3}
\left|
 \begin{array}{lcl}
	\lambda^{(s)}_{1,i}{\bf e_{ii_s}}+ \lambda^{(s)}_{2,i}{\bf e_{ii_{s+1}}} +
	\lambda^{(s)}_{3,i}{\bf e_{ii_{s+2}}} & = & 0\\[1ex]
	\lambda ^{(s)}_{1,i}+ \lambda ^{(s)}_{2,i}+\lambda ^{(s)}_{3,i}
	& = & 1,
\end{array}
\right.
\ee
where ${\bf e_{ii_{s+k}}}$ is the unit vector on the edge $e_{ii_{s+k}}$ starting from $V_i$ to $V_{i_{s+k}},
\ k=0,1,2$.

Note that basic curve networks are associated with points that have
at least three edges incident to them.
We denote by
$N_B$ the set of pairs of indices $is$ for which a basic curve
network is defined, i.~e.,
$$ N_B:=\{is\ |\ m_i\geq 3,\ i=1,\dots, n,\ s=1,\dots, m_i-2\}.$$

With each basic curve network $B_{is}$ for $is\in N_B$ we associate
a number $d_{is}$ defined by
$$
d_{is}=\frac{\lambda^{(s)}_{1,i}}{\| e_{ii_s}\| }
(z_{i_s}-z_i)+\frac{\lambda^{(s)}_{2,i}}{\| e_{ii_{s+1}}\|
}(z_{i_{s+1}}- z_i)+ \frac{\lambda^{(s)}_{3,i}}{\| e_{ii_{s+2}}\|
}(z_{i_{s+2}}-z_i),
$$
which reflects the position of the data in the supporting set of
$B_{is}$.

The following
three lemmas are proved in \cite{AEIV} for $p=2$ but they clearly hold for any $p$, $1<p\leq\infty$.
\begin{Lem}\label{lem2}
	\label{lemma1}
	Functions $B_{is},\ is\in N_{B}$, are linearly independent in $E$.
\end{Lem}

\begin{Lem}\label{lem3}
	{$F\in {\cal C}_p(E) \Leftrightarrow \lambda^{(s)}_{1,i}f_{ii_s}'(0)+
		\lambda^{(s)}_{2,i}f_{ii_{s+1}}'(0)+\lambda^{(s)}_{3,i}f_{ii_{s+2}}'(0)=0,\ is\in N_B$}
\end{Lem}

\begin{Lem}\label{lem4}
	{$F\in {\cal C}_p(E) \Leftrightarrow \langle
		F'',B_{is}\rangle=d_{is},\ is\in N_B$}
\end{Lem}

In \cite{V4} a full characterization of the solution to the extremal problem $(P_p)$
for $1<p<\infty$ was provided. Finding it comes down to the unique solution of a system of equations.
The following theorem was established in \cite{V4}.
\begin{Thm}\label{th1}
	In the case of strictly convex data problem ${({P}_p)}$, $1<p<\infty$, has a unique solution $F^*$. The second derivative of the solution $F^{*\prime\prime}$ has the form
	$$F^{*\prime\prime}=\left(\sum_{is\in N_B}\alpha_{is}B_{is}\right)_+^{q-1}$$
	where $1/p+1/q=1$, $(x)_+:=\max (x,0)$ and the coefficients $\alpha_{is}$ satisfy the following nonlinear system of equations
	\be{e22}
	\int_E \left(\sum_{is\in N_B}\alpha_{is}B_{is}\right)_+^{q-1}B_{kl}dt=d_{kl},\ \mbox{for}\ kl\in
	N_B.
\ee
\end{Thm}

\section{The solution for $p=\infty$}\label{s2}

In this section we consider the case $p=\infty$ and the corresponding extremal problem

$$({\bf P}_{\infty})\quad  \mbox{\it\ Find\ }F^*\in {\cal C}_{\infty}(E)\  \mbox{\it\
	such that }\| F^{*\prime\prime}\|_{\infty}=\inf_{F\in {\cal C}_{\infty}(E)}\|
F''\|_{\infty},
$$
where the class ${{\cal C}}_{\infty}(E)$ is defined by (\ref{e100}).
First, we show that the problem has a solution whose restriction on every edge $e\in E$ is a quadratic spline with at most one knot in the interval
$(0,\|e\|)$.

The interpolation curve network
$F=\{f_e\}_{e\in E}$ is smooth if and only if for each vertex $V_i$ of degree $m_i\geq 3$ the tangent vectors at the point $P_i$ to the curves incident to $P_i$ are coplanar or, which is the same, any three of them are coplanar. According to Lemma~\ref{lem3} this holds if and only if
$$
\displaystyle { \lambda^{(s)}_{1,i}f'_{ii_s}(0)+\lambda^{(s)}_{2,i}f'_{ii_{s+1}}(0)+\lambda^{(
		s )} _{3,i}f'_{ii_{s+2}}(0)=0,\ is\in {N}_B},
$$
where $\lambda ^{(s)}_{1,i},\ \lambda^{(s)}_{2,i},\ \lambda^{(s)}_{3,i}$ are defined as the unique solution to system (\ref{e3}). Hence, the class ${{\cal C}}_{\infty}(E)$ can be characterized in the following way:
\begin{eqnarray*}
	{{\cal C}}_{\infty}(E)\,:=\,\{ F\,:\, F(V_i)=z_i,\ i=1,\dots ,n,\
	f'_e\in AC,\ f''_e\in L_{\infty},\ f''_e\geq 0\ \mbox{a. e.},\ e\in
	E,\\[1ex]
	\lambda^{(s)}_{1,i}f'_{ii_s}(0)+\lambda^{(s)}_{2,i}f'_{ii_{s+1}}(0)+\lambda^{(
		s )} _{3,i}f'_{ii_{s+2}}(0)=0,\ is\in {N}_B\}.
\end{eqnarray*}

If $F\in {{\cal C}}_{\infty}(E)$ then  $F'$ belongs to the following class of curve networks
\begin{eqnarray*}
	{{\cal C}}'_{\infty}(E)\,=\,\{ G=\{g_e\}_{e\in E}\,:\,
	g_e\in AC,\ g'_e\in L_{\infty},\ g_e\ \mbox{is monotonically increasing},\\[1ex]
	\int_0^{\|e\|}g_e(t)dt=z_j-z_i,\ e=e_{ij}\in E,\ 	\lambda^{(s)}_{1,i}g_{ii_s}(0)+\lambda^{(s)}_{2,i}g_{ii_{s+1}}(0)+\lambda^{(
		s )} _{3,i}g_{ii_{s+2}}(0)=0,\ is\in {N}_B\}.
\end{eqnarray*}

On the other hand, if $G\in {{\cal C}}'_{\infty}(E)$ then the primitive
$F$ of $G$ for which $F(V_1)=z_1$ belongs to the class ${{\cal
		C}}_{\infty}(E)$. Therefore the problem $({P}_{\infty})$ is equivalent to the following problem
	$$({\bf P}'_{\infty})\quad  \mbox{\it\ Find\ }G^*\in {\cal C}'_{\infty}(E)\  \mbox{\it\
		such that }\| G^{*\prime}\|_{\infty}=\inf_{G\in {\cal C}'_{\infty}(E)}\|
	G'\|_{\infty}.
	$$

Next we show that the problem $({P}'_{\infty})$ has a solution in the following class
\be{e337}
	{\cal L}_{\infty}(E):=\{G=\{g_e\}_{e\in E}\,:\,G\in
	{{\cal C}}'_{\infty}(E),\ g_e\ \mbox{\it is\ piecewise\ linear}\ \mbox{\it with\ at\ most\
		one\ knot\ in}\ (0,\|e\|),\ e\in E\}.
\ee

Let $A,\,a,\,b,\,c$ be real numbers, such that $c>0$, $a<b$ and $A<(a+b)c$. We denote by $M$ the set of functions defined in $[0,c]$ that satisfy the following four properties:
\begin{eqnarray*}
	(i)\ &&\ g\in AC,\ g'\in L_{\infty},\\
	(ii)\ &&\ g(0)=a,\,g(c)=b,\\
	(iii)\ &&\ \int_0^cg(t)dt=A,\\
	(iv)\ &&\ g\ \mbox{is monotonically increasing}.
\end{eqnarray*}
The following lemma holds.
\begin{Lem}\label{lem5} ({Iliev and Pollul}\footnote{This lemma has been formulated by {Iliev and Pollul} in \cite{IP2} for the case
		$a=0,\,b=c=A=1$ without proof.}) {In the class $M$ there exists a unique function 	
		$g^*$ such that	$\|g^{*\prime}\|_{\infty}=\inf_{g\in M}\|g'\|_{\infty}$. The function
		$g^*$ is piecewise linear with at most one knot in the interval $(0,c)$.}
\end{Lem}

\noindent
\begin{proof} First we consider the case $A\leq (a+b)c/2$. Let $t_0=((a+b)c-2A)/(b-a)$ and
$$
g^*(t)=\left\{
\begin{array}{cc}
	a & \mbox{for}\ 0\leq t\leq t_0\\  & \\
	a+(t-t_0)(b-a)/(c-t_0)\qquad & \mbox{for}\
	t_0< t\leq c.
\end{array}
\right.
$$
Function $g^*$ belongs to $M$ and $\|g^{*\prime}\|_{\infty}=
{\displaystyle \frac{(b-a)^2}{2(A-ac)}}$. For arbitrary function
$g\in M$ we obtain from {H\"{o}lder}'s inequality
\begin{eqnarray*}
	\frac{(b-a)^2}{2}&=&|\int_0^c(g-a)g'|\leq \|g-a\|_1\,\|g'\|_{\infty}\\[1ex]
	&=&\|g'\|_{\infty}\int_0^c(g(t)-a)dt=(A-ac)\,\|g'\|_{\infty}
\end{eqnarray*}

Hence, ${\displaystyle
	\frac{(b-a)^2}{2(A-ac)}\leq\|g'\|_{\infty}}$, i.e.
$\|g^{*\prime}\|_{\infty}\leq\|g'\|$.
\vspace{1ex}

Now we consider the case $A>(a+b)c/2$. Let $t_0=2(bc-A)/(b-a)$ and
$$
g^*(t)=\left\{
\begin{array}{cc}
	a+t(b-a)/t_0\qquad & \mbox{for}\ 0\leq
	t\leq t_0\\  & \\
	b & \mbox{for}\ t_0<
	t\leq c. \end{array}
\right.
$$
The function $g^*$ belongs to $M$ and
$\|g^{*\prime}\|_{\infty}={\displaystyle\frac{(b-a)^2}{2(bc-A)}}$.
For arbitrary $g\in M$ from {H\"{o}lder}'s inequality we have
$$
\frac{(b-a)^2}{2}=|\int_0^c(b-g)g'|\leq \|b-g\|_1\,\|g'\|_{\infty}
=(bc-A)\|g'\|_{\infty}.
$$
Hence, ${\displaystyle \frac{(b-a)^2}{2(bc-A)}\leq\|g'\|_{\infty}}$,
i.e. $\|g^{*\prime}\|_{\infty}\leq\|g'\|_{\infty}.$

For $A\not=(a+b)c/2$ the function $g^*$ is piecewise linear with one knot in the interval
$(0,c)$, and for $A=(a+b)c/2$, $g^*$ is a linear function.
\end{proof}

\begin{Lem}\label{lem6}{\it It holds that
		$$
		\inf_{G\in {{\cal C}}'_{\infty}(E)}\|
		G'\|_{\infty}=\inf_{G\in {{\cal L}}_{\infty}(E)}\|
		G'\|_{\infty},
		$$
		where the class ${\cal L}_{\infty}(E)$ is defined by (\ref{e337}).
	}
\end{Lem}

\begin{proof} According to Corollary~\ref{cor1} the class ${\cal
	C}_{\infty}(E)$ is nonempty, hence ${\cal
	C}'_{\infty}(E)$ is nonempty too. Let $G\in {{\cal C}}'_{\infty}(E)$.
We apply Lemma \ref{lem5} for every edge $\ e=e_{ij}$ in $E$ for $c=\|e\|,\ a=g_e(0),\
b=g_e(\|e\|)$, $A=z_j-z_i$ and we obtain curve network  $G_1\in {\cal
	L}_{\infty}(E)$ such that $\|G_1'\|_{\infty}\leq\|G'\|_{\infty}$.
\end{proof}

\begin{Thm}\label{th2}{The problem $({P}'_{\infty})$ has a solution in the class ${\cal L}_{\infty}(E)$.}
\end{Thm}

\begin{proof}  Let  $l=\inf_{G\in {\cal L}_{\infty}(E)}\|G'\|_{\infty}$ and
$G_{\nu}=\{g_{\nu,\,e}\}_{e\in E}$, $\nu =1,\dots$ be a sequence of functions in
${\cal L}_{\infty}(E)$ such that
$\lim_{\nu\rightarrow\infty}\|G'_{\nu}\|_{\infty}=l$. Let $e$ be a fixed edge in $E$.
From Lemma~\ref{lem6} it follows that the functions $g_{\nu,e},\ \nu=1,2,\dots$ are piecewise linear with at most one knot in $(0,\|e\|)$ and are completely determined by their values $g_{\nu,\,e}(0)$ and
$g_{\nu,\,e}(\|e\|)$ which in turn form a bounded sequence when
$\nu =1,2,\dots$ . Therefore the sequence $\{g_{\nu ,e}\},\
\nu=1,2,\dots$ has a convergent subsequence with  limit function $g^*_e$ as $\nu\rightarrow\infty$.
The function $g_e^*$ is monotonically increasing, piecewise linear with at most one knot in $(0,\|e\|)$, and
$\int_0^{\|e\|}g^*_e(t)dt=z_j-z_i$. By choosing consecutively convergent subsequences  on each edge $e\in E$ we construct a convergent subsequence
$\{G_{\nu}\}_{\nu=1}^{\infty}$. Let us denote its limit by
$G^*=\{g^*_e\}_{e\in E}$. We have $\|G^{*\prime}\|_{\infty}=l$.
To prove that $G^*$ belongs to the class ${\cal L}_{\infty}(E)$ it suffices to verify the smoothness conditions
$
\lambda^{(s)}_{1,i}g^*_{ii_s}(0)+\lambda^{(s)}_{2,i}g^*_{ii_{s+1}}(0)+\lambda^{
( s )} _{3,i}g^*_{ii_{s+2}}(0)=0,\ is\in {N}_B.
$

The latter follows from  the limit transition in the corresponding smoothness conditions for the curve networks$\{G_{\nu}\}_{\nu=1}^{\infty}$,
$
\lambda^{(s)}_{1,i}g_{\nu,ii_s}(0)+\lambda^{(s)}_{2,i}g_{\nu,ii_{s+1}}(0)+
\lambda^{ ( s )} _{3,i}g_{\nu,ii_{s+2}}(0)=0,\ is\in {N}_B.
$
\end{proof}

\begin{Cor}\label{cor2}{\it The extremal problem $({P}_{\infty})$ has a solution whose restriction on every edge $e\in E$ is a quadratic spline with at most one knot in the interval $(0,\|e\|)$.}
\end{Cor}

\begin{Thm}\label{th3}
{Let there exist $C>0$ such that the system
		\begin{equation}\label{e222}
		\int_EC \left( \sum_{is\in N_B}\alpha_{is}B_{is}\right)^0_+
		B_{kl}=d_{kl},\ kl\in N_B
	\end{equation}
has a solution for $\alpha_{is}\in\mathbb{R}$, $is\in N_B$. Then the function	
		$F^*\in  {\cal C}_{\infty}(E)$ such that
	\be{e101}
	F^{*\prime\prime}=C\left(
	\sum_{is\in {\cal N}_B}\alpha_{is}B_{is}\right)^0_+
	\ee
		solves the problem $( {P}_{\infty})$. }
\end{Thm}

\begin{proof}
	Let $F=\{f_e\}_{e\in {\cal N}_B}$
be an arbitrary curve network from ${{\cal C}}_{\infty}(E)$.
From {H\"{o}lder}'s inequality we obtain
\be{e19}
\|F''\|_{\infty}\|(\sum_{is\in {\cal N}_B}\alpha_{is}B_{is})_+
\|_1 \geq \langle\ F'',(\sum_{is\in {\cal
	N}_B}\alpha_{is}B_{is})_+ \rangle\ . \ee

Since $f_e''\geq 0$ for $e\in E$ then
\be{e20}
\langle F'',(\sum_{is\in {\cal N}_B}\alpha_{is}B_{is})_+ \rangle \geq\
\langle F'',
\sum_{is\in {\cal N}_B}\alpha_{is}B_{is} \rangle =\ \sum_{is\in {\cal N}_B}\alpha_{is}\langle F'',B_{is}\rangle .
\ee
Since $F^*,F\in  {\cal C}_{\infty}(E)$ then from Lemma~\ref{lem4}
it follows
$$\langle F'',B_{is}\rangle=d_{is}=\langle F^{*\prime\prime},B_{is}\rangle,
$$  and we obtain consecutively
\begin{eqnarray}
\sum_{is\in {\cal N}_B}\alpha_{is} \langle F'',B_{is}\rangle &=&
\sum_{is\in {\cal N}_B}\alpha_{is}\langle
F^{*\prime\prime},B_{is}\rangle =\langle F^{*\prime\prime},
\sum_{is\in {\cal N}_B}\alpha_{is}B_{is}\rangle \nonumber\\
&&\label{e21}\\[1ex]
&=&\langle
\ C\ (
\sum_{is\in {\cal N}_B}\alpha_{is}B_{is})^0_+,
\sum_{is\in {\cal N}_B}\alpha_{is}B_{is}\rangle
=C\ \|(\sum_{is\in {\cal N}_B}\alpha_{is}B_{is})_+\|_1.\nonumber
\end{eqnarray}

From (\ref{e19}) and (\ref{e20}) it follows $\|F^{\prime\prime}\|_{\infty}\geq C$.
Since $\|F^{*\prime\prime}\|_{\infty}=C$ then
$F^*$ solves the problem $( {P}_{\infty})$.
\end{proof}

\begin{rmk}\label{zab1}
{Coefficients $\alpha_{is}$ are determined up to a constant factor.}
\end{rmk}

\section{Example}\label{s3}

As an illustration of our theoretical results established in Section~\ref{s2}, here we present and discuss the solution for $p=\infty$ to a simple example.

\begin{Exmp}\label{exmp1}
We consider data obtained from a regular triangular pyramid. They are $P_1=(-1/2,-\sqrt{3}/6,0)$, $P_2=(1/2,-\sqrt{3}/6,0)$, $P_3=(0,\sqrt{3}/3,0)$, and $P_4=(0,0,-1/2)$.
The vertices of the associated triangulation $T$ are $V_1=(-1/2,-\sqrt{3}/6)$, $V_2=(1/2,-\sqrt{3}/6)$, $V_3=(0,\sqrt{3}/3)$, and $V_4=(0,0)$. The set of indices defining the edges of $T$ is $N_B\,=\,\{ 12,23,31,41,42,43\}$. The solution $F^*=\{f^*_{ij}\}_{ij\in N_B}$ for $p=\infty$ whose restriction on every edge is a quadratic spline with at most one knot can be computed directly. It is
$$
\begin{array}{c}
	f^*_{12}(t)=f^*_{23}(t)=f^*_{31}(t)=3(t^2-t)/2, \ 0\leq t\leq 1;\\[1ex]
	f^*_{i4}(t)=3t^2/2-\sqrt{3}t,\ 0\leq t\leq \sqrt{3}/3,\quad i=1,2,3.
\end{array}
$$
The $L_{\infty}$-norm of the second derivative of $F^*$ is $\|F^{*\prime\prime}\|_{\infty}=3$. The triangulation $T$ and the corresponding edge convex minimum $L_{\infty}$-norm network are shown in Figure~\ref{f2}.

We note that in this case $F^{*\prime\prime}$ can not be represented in the form (\ref{e101}). Hence, the corresponding system (\ref{e222}) has no solution.

	\end{Exmp}

\begin{figure}
	\centering
	\begin{minipage}[b]{2.in}
		\includegraphics[width=.85\textwidth]{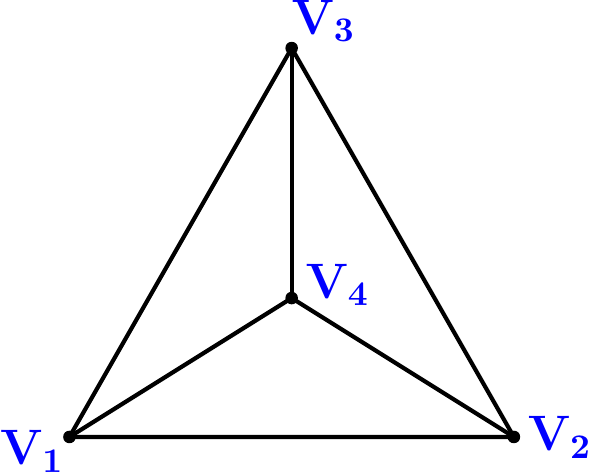}
	\end{minipage}
	~~~~~~~~~
	\begin{minipage}[b]{2.2in}
			\includegraphics[width=.95\textwidth]{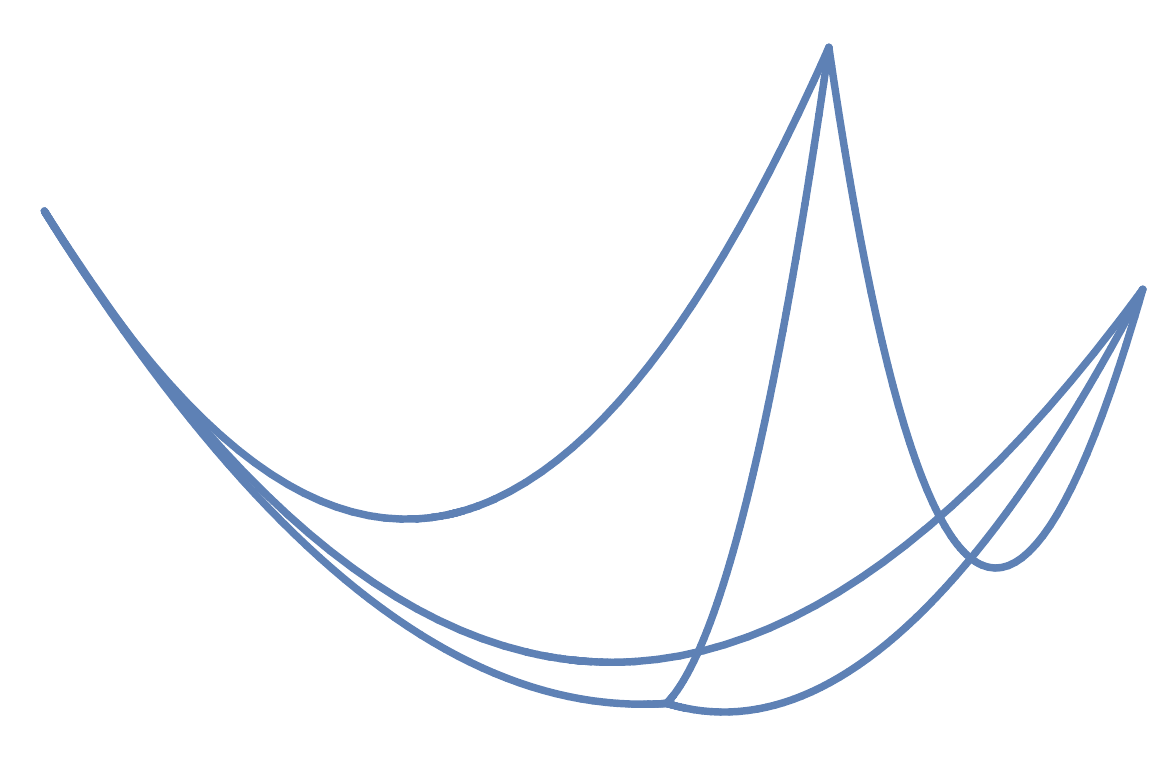}
		\centering
		{${p=\infty,\ \|F^{*\prime\prime}\|_{\infty}=3.}$}
	\end{minipage}
	\caption{Example~\ref{exmp1}: (left) the triangulation $T$; (right) the corresponding edge convex minimum $L_{\infty}$-norm network $F^*$}\label{f2}
\end{figure}

\section{Conclusions}\label{s4}

We have shown the existence of a solution to the extremal problem for interpolation of convex scattered data in $\mathbb{R}^3$ using edge convex minimum $L_p-norm$ networks for $p=\infty$. We characterized it at each edge as a quadratic spline with at most one knot whose second derivative is either zero, or constant C>0. The question about the uniqueness of the solution remains open although the solution for related problems in the univariate case is not unique, see \cite{IP1} for the constrained problem and also \cite{deB1} for the unconstrained problem.

The solution for $p=\infty$ for given data can be obtained by solving the system (\ref{e222}) in the case where it has a solution. In the case $1<p<\infty$ the corresponding system \ref{e22} has been solved by a Newton-type algorithm \cite{V3}. It is not clear if this approach works here.


\section*{Acknowledgments}
This work was supported in part by Sofia University Science Fund Grant No. 80-10-109/2022

\nocite{*}
\bibliography{ARXIV2022bibfile}

\end{document}